\documentclass{amsart}

\usepackage{amssymb}
\usepackage{amsmath}
\usepackage{amsthm}

\usepackage{mathtools}
\usepackage{enumitem}

\usepackage{todonotes}
\usepackage{color,xcolor}

\usepackage{subcaption}

\newtheorem{theorem}{Theorem}[section]
\newtheorem{lemma}[theorem]{Lemma}
\newtheorem{corollary}[theorem]{Corollary}
\newtheorem{proposition}[theorem]{Proposition}

\theoremstyle{definition}
\newtheorem{definition}[theorem]{Definition}

\theoremstyle{remark}
\newtheorem{remark}[theorem]{Remark}

\numberwithin{equation}{section}

\newcommand{\N}{\mathbb{N}}

\newcommand{\R}{\mathbb{R}}

\newcommand{\Rn}{\R^n}

\newcommand{\tr}{\operatorname{tr}}

\renewcommand{\div}{\operatorname{div}}

\newcommand{\Fcrit}{\tau_c}
\newcommand{\FT}{\tau_0}
\newcommand{\dens}{\rho}

\begin{document}
\title[Taylor Scaling]{A proof of Taylor scaling for curvature-driven dislocation motion through random arrays of obstacles}

\author{Luca Courte}
\address[Luca Courte]{Abteilung f\"ur Angewandte Mathematik,
	Albert-Ludwigs-Universit\"at Freiburg, Raum 228, Hermann-Herder-Straße 10, 79104 Freiburg i. Br.}
\email{luca.courte@mathematik.uni-freiburg.de}
\urladdr{https://aam.uni-freiburg.de/mitarb/courte/}

\author{Patrick Dondl}
\address[Patrick Dondl]{Abteilung f\"ur Angewandte Mathematik,
	Albert-Ludwigs-Universit\"at Freiburg, Raum 217, Hermann-Herder-Straße 10, 79104 Freiburg i. Br.}
\email{patrick.dondl@mathematik.uni-freiburg.de}
\urladdr{https://aam.uni-freiburg.de/agdo/}

\author{Michael Ortiz}
\address[Michael Ortiz]{California Institute of Technology (MC 105-50), Pasadena, CA 91125}
\email{ortiz@aero.caltech.edu}
\urladdr{http://www.ortiz.caltech.edu/}



\date{\today}

\begin{abstract}
We prove Taylor scaling for dislocation lines characterized by line-tension and moving by curvature under the action of an applied shear stress in a plane containing a random array of obstacles. Specifically, we show--in the sense of optimal scaling--that the critical applied shear stress for yielding, or percolation-like unbounded motion of the dislocation, scales in proportion to the square root of the obstacle density. For sufficiently small obstacle densities, Taylor scaling dominates the linear-scaling that results from purely energetic considerations and, therefore, characterizes the dominant rate-limiting mechanism in that regime.
\end{abstract}

\maketitle

\section{Introduction}

{\sl Taylor scaling}, i.~e., a power-law dependence
\begin{equation}\label{I0VsvO}
    \Fcrit \sim \sqrt{\dens}
\end{equation} %
of the flow stress (or over-stress in the presence of lattice friction) $\Fcrit$ for activation of plastic slip in a single crystal on some appropriate dislocation density $\dens$ is one of the accepted tenets of physical metallurgy (cf., e.~g., \cite{hirth:1968, hull:1984}). The precise measure $\dens$ of dislocation density depends on the rate-limiting mechanism under consideration. For instance, in his seminal work, Taylor \cite{edtm,j1l} sought to estimate the stress $\Fcrit$ required to break an edge dislocation dipole, and more generally regular lattices of parallel dislocations, and showed that $\Fcrit \sim 1/l$, with $l$ a measure of dislocation spacing. Taylor scaling of the form \eqref{I0VsvO} is then obtained by noting that the line dislocation density per unit volume in dislocation lattices scales as $\dens \sim l^{-2}$.

Taylor scaling also arises in theories of {\sl forest hardening} based on the line tension approximation (cf., e.~g., \cite{hull:1984}). In these theories, the aim is to estimate the increase in the flow stress for the activation of a slip system, or primary system, due to dislocations belonging to other systems, or secondary systems. Such secondary or 'forest' dislocations pierce through the primary slip planes, and pin down the primary dislocations moving in those planes through a number of mechanisms such as jog formation and other dislocation reactions. Often, forest dislocations are idealized as impenetrable point obstacles opposing the motion of the primary dislocations. Under these assumptions, simple line-tension estimates for the flow stress required to bypass a pair of pinning obstacles gives the relation $\Fcrit \sim 1/l$, with $l$ the distance between the pinning points. A meanfield type argument \cite{kocks:1966, ortiz:1982} then yields a flow stress of the for \eqref{I0VsvO}, with $\dens$ the density of secondary or forest dislocations and $1/\sqrt{\dens}$ supplying an estimate of the mean distance between pinning points. 

More detailed numerical treatments of the forest hardening mechanism, such as the seminal calculations of Foreman and Makin \cite{Foreman.1966, Foreman.1967}, account for the percolation-like motion of dislocation lines in the primary slip plane through random arrays of point obstacles. In these calculations, which are based on the line-tension approximation, parts of the dislocation line are observed to become de-pinned upon an increase in the resolved shear stress and to jump to more stable obstacles, where they become pinned again. In so doing, the dislocation line sweeps through a certain area of the slip plane. The macroscopic volume average of all such slip areas gives the incremental slip strain in units of Burgers vector. For point obstacles obeying a Poisson distribution on the slip plane, the calculations of Foreman and Makin \cite{Foreman.1966, Foreman.1967} show that the dislocation lines travel an increasing distance, resulting in increasing incremental slip strains, as the applied resolved shear stress is increased. Eventually, a critical, or {\sl percolation}, value $\Fcrit$ of the applied resolved shear stress is attained at which no equilibrium configuration exists and the dislocation jumps become unbounded. Foreman and Makin \cite{Foreman.1966, Foreman.1967} observed numerically that $\Fcrit$ appeared to scale as $1/l$, with $l$ a measure of the distance between point defects. This relation in turn again gives Taylor scaling \eqref{I0VsvO} upon observing that $l \sim 1/\sqrt{\dens}$, with $\dens$ the density of forest dislocation line per unit volume.

In this article, we prove Taylor scaling for the dislocation tension model of forest hardening by a random array of obstacles in the sense of {\sl optimal scaling}, i.~e., we show that there is a constant $C>0$ so that
\[
\frac{1}{C}\sqrt{\dens}  \le \Fcrit - \FT \le C\sqrt{\dens},
\]
where $\FT$ accounts for lattice friction. Our setting is that of a line-tension model for the motion-by-curvature of dislocations in a given slip plane, where forest dislocations -- acting as small, disc shaped obstacles -- are distributed according to a two-dimensional Poisson point process with intensity $\dens$. This result is stated in detail in Corollary \ref{cor:scaling}. We recall that optimal scaling laws are established buy producing upper and lower bounds of a power-law type with matching exponents for all parameters in both bounds. 

Optimal scaling methods were pioneered by Kohn and M\"uller \cite{KM92} as part of their seminal work on branched structures in martensite, and have been since successfully applied to a number of related problems, including shape-memory alloys, micromagnetics, crystal plasticity, and others \cite{KM92, KM94, CKO99, Conti00, ContiOrtiz05}. We also note that, whereas the line tension approximation pervades the better part of physical metallurgy, (cf., e.~g., \cite{hirth:1968, hull:1984}), rigorous results showing that line tension indeed describes the energy of sufficiently dilute dislocations, or, equivalently, dislocations of sufficiently small core radius, have only recently become available \cite{Garroni.20066n, Conti:2015, Fonseca.2021}.

It bears emphasis that the results presented in this paper stand in contrast to previous estimates of the critical resolved shear stress that follow from the work of Garroni and M\"uller \cite{Garroni.20066n, Garroni.2005}. These differences stem directly from variances in the treatment of the problem and the definition of critical resolved shear stress. Thus, the present approach aims to analyze the motion of dislocations through random arrays of obstacles using a kinetic, or dynamic, formulation of the problem. Based on this formulation, we seek viscosity solutions with the aid of a comparison principle (cf., e.~g., \cite{CourteDondl}). Criticality in this context is identified with the failure of a stationary viscosity supersolution to exist. Conditions for the existence of such supersolutions, equivalent to Taylor scaling, are given in Theorem~\ref{thm:existence_of_subsupersol}. By contrast, Garroni and M\"uller \cite{Garroni.20066n, Garroni.2005} study the energetics of a Peierls-Nabarro model proposed by Koslowski {\sl et al.} \cite{KCO}. In this energetic setting, criticality is associated with a change of the energy minimizing state. This notion of criticality in the homogenization limit of \cite{Garroni.20066n, Garroni.2005} results in a scaling of the critical resolved shear stress proportionally to the density of defects, each defect contributing in accordance with its own capacity.

This discrepancy exemplifies the difference in problems belonging to the ``wiggly class'' \cite{Abeyaratne.1996} between the gradient flow of the $\Gamma$-limit of the energy and the limit of the gradient flows of the energy. Indeed, such flows are only equivalent for problems in the ``Serfaty class'' \cite{Serfaty.2011}, for which $\Gamma$-limit and gradient flow commute. For the problem considered here, it is therefore important to consider critical states instead of energy minimizers \cite{DKW}. Indeed, for sufficiently small obstacle densities, as are likely to be found in practice, Taylor scaling dominates linear scaling and characterizes the dominant rate-limiting mechanism for dislocation motion.

The remainder of this work is organized as follows. In section \ref{sec:model}, we describe in detail our mathematical model and results
Section \ref{sec:comp} is devoted to the proof of a comparison principle for our model. The main theorem is finally proved in section \ref{sec:lower_bound} and \ref{sec:upper_bound} for the lower and upper scaling bound, respectively.

\subsection{The mathematical model and the main result} \label{sec:model}
Let $\Gamma(t)$ be a curve in $\R^2$ describing the motion of a curvature driven interface (our dislocation line) through a heterogeneous medium (to be precise, a homogeneous medium with random obstacles) where the obstacles exert an additional dry friction \cite{Dondl_20c}. In this setting, the equation of motion is given by
\begin{equation}\label{eq:main}
\mathcal{F}(v_n(\xi)) + \varphi(\xi) \partial R(v_n(\xi)) \ni \kappa(\xi) + F, \text{ for all } \xi \in \Gamma(t),
\end{equation}
where $v_n$ is the dislocation's normal velocity and $\mathcal{F} : \R \to \mathcal{P}(\R)\setminus\{\emptyset\}$ is a nonlinear, set-valued function, $R(v) = |v|$ and $\partial R$ its subdifferential, $\kappa$ the mean curvature, $\varphi:\R^2 \to [0, \infty)$ a random function describing the heterogeneous medium and $F > 0$ some external driving force. The kinetic relation $\mathcal{F}$ describes how the exerted force gets translated to the velocity of the interface. It is reasonable to assume that $\mathcal{F} = \partial\Psi$ where $\Psi$ is a dissipation potential and $\partial \Psi$ its subdifferential.

Our model includes the possibility for a stick-slip kinetic relation to account for lattice friction. The forest dislocations (obstacles) always oppose the motion of the interface crossing them, as the additional dry friction from the $\varphi(\xi) \partial R(v_n(\xi))$-term requires a fixed amount of energy to overcome. Our results remain valid if the kinetic relation is replaced by the simple viscous one $\mathcal{F}(v)=v$.

Equations of curvature driven interface motion are commonly used to describe different phenomenona arising in materials science, e.g., phase propagation and dislocation motion. It is therefore not surprising that  curvature flow in random media has received significant attention in the last years \cite{Coville:2009uv, Dirr:2008ki, Dirr:2006ui, Dondl_11b, Armstrong.2018}. Linearized and linear, non-local models were investigated in \cite{Dondl_12a, Bodineau:2013ur, Dondl_16b,Dondl_19b, Dondl:2015kk, Dondl_17c, Dondl_20c}.

A curve evolving according to \eqref{eq:main} need not remain an immersed curve for all time even when starting with a flat initial curve (which is properly immersed) as a pinch off may occur. It might be the easiest to imagine an obstacle that offers a large amount of resistance in a positive set around its core. The curve will wrap around the obstacle and pinch off, leaving a circle around the obstacle behind, instead of moving through the obstacle. To be able to deal with this change in topology, we will use the level-set formulation of \eqref{eq:main}.

Let $\Gamma(t)$ be an interface and $d_s(\cdot, t)$ its signed distance function. We use the convention $d_s(\cdot, t) \le 0$ in the interior of $\Gamma(t)$ and if $\Gamma(t)$ splits $\Rn$ in two, then we assume (after applying a rotation) that $d_s(x, y, t) \le 0$ if $y \to +\infty$. We define $\Omega_+(t) \coloneqq \{ d_s(t) > 0 \}$ as the interior set of the interface and $\Omega_-(t) \coloneqq \{ d_s(t) < 0 \}$. Moreover, it is reasonable to assume enough regularity on $\Gamma(t)$ such that $\Rn = \Omega_-(t) \cup \Omega_+(t) \cup \Gamma(t)$ and that $\partial \Omega_-(t) = \partial \Omega_+(t) = \Gamma(t)$. Finally, the normal of $\Gamma(t)$ is chosen to be the outward normal of $\Omega_+$. Now, for a general function $u : \R^2\times I \to \R$ such that $u(\cdot, t) < 0 \iff d_s(\cdot, t) < 0$ and $u(\cdot, t) > 0 \iff d_s(\cdot, t) > 0$. The zero level-set of $u(\cdot, t)$ is a level-set solution of \eqref{eq:main} if $u$ satisfies
\begin{equation}\label{eq:main_levelset}
|\nabla u| \mathcal{F}(\tfrac{u_t}{|\nabla u|}) + \varphi(\cdot) \eta_{\theta}(\tfrac{u}{|\nabla u|}) |\nabla u| \partial R(u_t) \ni |\nabla u|\div(\tfrac{\nabla u}{|\nabla u|}) + |\nabla u| F,
\end{equation}
where $\eta_{\theta}$ is a radially symmetric $C^\infty(\R)$ function with $\eta_{\theta}(0) = 1$, $\eta_{\theta}$ is decreasing on $[0, \infty)$, and $\eta_{\theta}(s) = 0$ for $s \ge \theta$.

Our results are valid as long as \eqref{eq:main_levelset} satisfies a comparison principle and a unique viscosity solution exists (see \cite{CourteDondl} for a discussion on viscosity solutions for partial differential inclusions). We refer also to section \ref{sec:comp} where we prove existence and that a comparison principle holds for $\mathcal{F}(a) = a + \tau \partial R(a)$ with $\tau \ge 0$.

Before continuing the discussion let us fix the shape of the random function $\varphi$. In line with our setting of a dislocation subject to line tension in a field of obstacles, we first define an obstacle set consisting of small discs centered around the the points in a Poisson process. The additional dry friction $\varphi$ when crossing obstacles is then the characteristic function of the obstacle set, after some mollification and multiplication by a strength-prefactor.
\begin{definition}[Obstacles]\label{def:obstacles}
	Let $r_1 > r_0 > 0$ and define the random set of obstacles
	\[
	O \coloneqq \bigcup_{i=0}^\infty B_{r_0+{\frac{r_1-r_0}{2}}}(x - x_i(\omega), y - y_i(\omega))
	\]
	where the tuples $x_i, y_i$ are generated by a $2$-dimensional Poisson point process of intensity $\dens > 0$.
	The random function $\varphi : \R \times \R \times \Omega \to [0, \infty)$ is then given by
	\[
	\varphi(x, y, \omega) = - f \eta_{\frac{r_1-r_0}{2}} * \chi_{O}(x, y),
	\]
	with $f \ge 0$ the maximal pinning force of the obstacles and $\eta_{\frac{r_1-r_0}{2}}$ is the rescaled standard mollifier with support in $B_{\frac{r_1-r_0}{2}}(0)$.
\end{definition}

\begin{remark}\label{rk:mean_site_distance}
	The mean distance between a site $(x_i(\omega), y_i(\omega))$ and the closest next site is proportional to $\frac{1}{\sqrt{\dens}}$.
	
\end{remark}

For equation \eqref{eq:main_levelset}, let us define the quantities $\underline{F}_{\rm crit}, \overline{F}_{\rm crit}$. They are the critical forces such that
\begin{itemize}
	\item whenever $F \ge \underline{F}_{\rm crit}$ any solution is ballistic, i.~e., the dislocation passes over a strictly positive mean area per time unit,
	\item whenever $F \le \overline{F}_{\rm crit}$ any solution gets pinned, i.~e., the dislocation remains bounded in the vicinity of its initial state for all time.
\end{itemize}
It is easy to see that there is a (non optimal) critical force $\underline{F}_{\rm crit} \in (0, \infty]$ such that whenever $F \ge \underline{F}_{\rm crit}$ there exists a ballistic (graphical) subsolution. As the solution to equation \eqref{eq:main_levelset} has to stay above this subsolution (due to the comparison principle), the solution is also ballistic. On the other hand, in \cite{Dondl_11b} it was shown that there is also a critical force $\overline{F}_{\rm crit} \in (0, \infty)$ such that for all $F \le \overline{F}_{\rm crit}$ there is a stationary (graphical) supersolution. It has been shown that there are cases where $\overline{F}_{\rm crit} < \underline{F}_{\rm crit}$, i.~e., there is a regime where sub-ballistic solutions exist \cite{Dondl_17c}. If comparison holds, the levelset solution cannot pass the zero levelset of the stationary supersolution and hence it becomes pinned.

The main result of this article is the following theorem whose proof can be found in sections \ref{sec:lower_bound} and \ref{sec:upper_bound}.
\begin{theorem}\label{thm:existence_of_subsupersol}
	If $\dens > 0$ and $\theta > 0$ are small enough, then almost surely equation \eqref{eq:main_levelset} admits
	\begin{enumerate}
		\item a stationary viscosity supersolution with linear growth whose zero level-set stays above $\{y=0\}$, if $F < c \sqrt{\dens} + \sup \mathcal{F}(0)$,
		\item a ballistic viscosity subsolution with linear growth whose initial zero level-set is $\{y=0\}$, if $F > C \sqrt{\dens} + \sup \mathcal{F}(0)$.
	\end{enumerate}
	The constants $c, C$ are deterministic constants independent of $\dens$.
\end{theorem}
\begin{proof}
	The $(1)$ part is proved in theorem \ref{thm:levelset_super} and the $(2)$ part in theorem \ref{thm:levelset_sub}.
\end{proof}

We immediately deduce the following.

\begin{corollary}[Scaling of the critical force] \label{cor:scaling}
	Assume that \eqref{eq:main_levelset} satisfies a comparison principle then it holds almost surely for an initially flat interface and for $\dens$ small enough that
	\[
	c \sqrt{\dens} + \sup \mathcal{F}(0) \le \overline{F}_{\rm crit} \le \underline{F}_{\rm crit} \le C \sqrt{\dens} + \sup \mathcal{F}(0),
	\]
	where $c, C$ are deterministic constants independent of $\dens$.
\end{corollary}
\begin{proof}
	The comparison principle allows to compare the viscosity sub- and supersolution from \eqref{thm:existence_of_subsupersol} with the unique solution. Therefore, we see that the bounds hold.
\end{proof}

Hence, we prove that the pinning to depinning transition scales like the inverse of the mean distance between nearest sites, see remark \ref{rk:mean_site_distance}. The proof is based on the construction of explicit sub- and supersolutions. Both constructions are novel though the construction for the lower bound is based on \cite{Dondl_11b}.

If $\mathcal{F} = \tau\partial R$, i.~e., the energy is dissipated only by dry friction kinetics, we no longer have unique or continuous solutions. As satisfying \eqref{eq:main_levelset} is a local condition, as soon as a discontinuous viscosity solution jumps, it can jump to almost any stationary curve satisfying the equation. In our given force field (noting its non-monotonicity), there are infinite many states for such stationary curves. We argue here that a physically reasonable solution for this purely rate-independent dissipation is the pointwise minimum of supersolutions, which is still a supersolution. We note that this minimal supersolution is also below the vanishing viscosity limit. We can thus still make the observation that any reasonable solution for the purely rate-independent case also follows the Taylor scaling law.

\begin{corollary}
	Let $\dens > 0, \theta > 0$ be small enough. If $F < c\sqrt{\dens} + \sup \mathcal{F}(0)$, then the zero level-set of the minimal supersolution $u_{\rm m}$, i.~e., the pointwise minimum of all supersolutions above the initial condition, of \eqref{eq:main_levelset} with $\mathcal{F} = \tau\partial R$  will almost surely remain bounded for all times.
\end{corollary}

\begin{proof}
	Due to the stability of viscosity solutions, $u_{\rm m}$ will be smaller then the vanishing viscosity solution $\liminf^*_{\epsilon \to 0} u_{\epsilon}$ where $\liminf^*_{\epsilon\to 0}$ is the half-relaxed limit (see \cite[Section~2.3]{CourteDondl}) and $u_{\epsilon}$ the unique viscosity solution (see section \ref{sec:comp}) to
	\begin{equation*}
	\mathcal{F}_{\epsilon}(v_n(\xi)) + \varphi(\xi) \partial R(v_n(\xi)) \ni \kappa(\xi) + F, \text{ for all } \xi \in \Gamma(t),
	\end{equation*}
	with $\mathcal{F}_{\epsilon}(v) \coloneqq \epsilon v + \tau \partial R(v)$ for all $v \in \R$. Let $v$ be the stationary supersolution constructed in section \ref{sec:lower_bound}, then we have for all $\epsilon > 0$ that almost surely for all $x \in \R^3$ and $t \in [0, T]$ it holds
	\[
		u_{\epsilon}(x, t) \le v(x, t).
	\]
	Taking the half-relaxed limit, now implies that almost surely
	\[
		u_{\rm m}(x, t) \le {\liminf_{\epsilon \to 0}}^* u_{\epsilon}(x, t) \le v(x, t).
	\]
	Hence, $u_{\rm m}$ stays below $v$ and the zero level-set of $u_{\rm m}$ remains bounded.
\end{proof}

\begin{remark}
This corollary shows that the minimal supersolution remains pinned as long as $F < c\sqrt{\dens} + \sup \mathcal{F}(0)$. On the other hand, there exist no (bounded) supersolutions if $F > C\sqrt{\dens} + \sup \mathcal{F}(0)$ as this would contradict the existence of a propagating subsolution as constructed in section \ref{sec:upper_bound}. Hence, we conclude that if $F > C\sqrt{\dens} + \sup\mathcal{F}(0)$ the solution jumps instantly to $+\infty$. In this sense, Corollary \ref{cor:scaling} is still valid in the purely rate-independent setting.
\end{remark}

\section{Comparison and Existence}\label{sec:comp}

In this section we show that \eqref{eq:main_levelset} satisfies a comparison principle under the additional assumption that $\mathcal{F}$ has the following form
\begin{equation}\label{eq:dissipation}
\mathcal{F}(a) \coloneqq a + \tau\partial R(a),
\end{equation}
where $\tau \ge 0$, and $\partial R$ is the subdifferential of the absolute value. %

We will first show that comparison holds and then prove existence using Perron's method. Using \eqref{eq:dissipation} we can rewrite \eqref{eq:main_levelset} in the following way
\[
u_t - |\nabla u|\div(\tfrac{\nabla u}{|\nabla u|}) - |\nabla u| F \in - (\tau + \varphi(\cdot) \eta_{\theta}(\tfrac{u}{|\nabla u|}))|\nabla u| \partial R(u_t).
\]

After applying an exponential rescaling in time, i.~e., applying for some $\dens > 0$ the map $u \mapsto e^{-\dens t} u$, it is enough to show comparison for the following equation
\begin{equation}\label{eq:levelset_rescaled}
F(u, u_t, \nabla u, D^2 u) \in G(x, u, \nabla u) \mathcal{S}(u_t + \dens u),
\end{equation}
with $F(r, a, p, X) \coloneqq a + \dens r - \tr(X) + \tr(\tfrac{p\otimes p}{|p|^2}X) - |p|F$, $G(x, r, p) \coloneqq (\tau + \varphi(x) \eta_{\theta}(\tfrac{r}{|p|}))|p|$, and $\mathcal{S}(a) \coloneqq -\partial R(a)$. The definition of viscosity solutions to equations like \eqref{eq:levelset_rescaled} have been stated in \cite{Dondl_20c} and \cite{CourteDondl}. Recall, that a subsolution of \eqref{eq:levelset_rescaled} has to be a subsolution of \eqref{eq:levelset_rescaled} with $F$ replaced by its lower-semicontinuous envelope $F_*$. Analogously a supersolution has to be a supersolution while replacing $F$ by its upper-semicontinuous envelope $F^*$. Introducing $\operatorname{MC}(p, X) \coloneqq - \tr(X) + \tr(\tfrac{p\otimes p}{|p|^2}X)$, we can rewrite $F$ as $F(r, a, p, X) \coloneqq a + \dens r +\operatorname{MC}(p, X) - |p|F$ and note that the application of the semicontinuous envelopes to $F$ results on applying it on $\operatorname{MC}$.

The following lemma will be useful for proving the comparison principle.
\begin{lemma}\label{lem:estimates}
	Let $\hat x, \hat y\in \R^2$,$a, b, u, v \in \R$, $p, q_1, q_2 \in \R^2$, $X, Y, A \in \operatorname{Sym}(2)$ with $a-b\ge \frac{\gamma}{T^2}$, then it holds
	\begin{align*}
	F^*&(v, b, p+q_1, Y-A) - F_*(u, a, p+q_2, X+A) \\ &
	\le -\tfrac{\gamma}{T^2} + \dens (v-u) + \operatorname{MC}^*(p+q_1, Y-A) - \operatorname{MC}_*(p+q_2, X+A) + |q_1-q_2|F,
	\end{align*}
	and
	\begin{align*}
	|G&(\hat y, v, p + q_1) - G(\hat x, u, p + q_2)| \\
	&\le \tau |q_1-q_2| + L_{\varphi} |p+q_1||\hat x - \hat y| +||\varphi||_\infty \left|\eta_{\theta}(\tfrac{v}{|p+q_1|})|p+q_1| - \eta_{\theta}(\tfrac{u}{|p+q_2|})|p+q_2|\right|,
	\end{align*}
	where $L_\varphi$ is the Lipschitz-constant of $\varphi$.
\end{lemma}
\begin{proof}
	These inequalities follow by direct computation.
\end{proof}

\begin{theorem}[Comparison Principle]
	Let $u$ be a subsolution to \eqref{eq:levelset_rescaled} and $v$ a supersolution to \eqref{eq:levelset_rescaled} with
	\begin{equation}\label{eq:linear_growth}
	u(x, t) - v(y, t) \le L (1+|x|+|y|) \text{ for all } (x, y, t) \in \R^2\times \R^2 \times (0, T)
	\end{equation}
	for some constant $L \ge 0$ and some final time $T < \infty$. If $u(\cdot, 0) \le v(\cdot, 0)$ and $\dens > 0$ is big enough then
	\[
	u(x, t) \le v(x, t) \text{ for all } (x, t) \in \R^2 \times [0, T).
	\]
\end{theorem}
\begin{proof}
	We begin the proof with a growth estimate.\\
	\textit{Step 1: Growth estimate} \\
	We show that the difference between $u$ and $v$ grows in the following way, there is a constant $C=C(L_{\eta_{\theta}}, L_\varphi, ||\varphi||_\infty, F, \tau) > 0$ that depends on the Lipschitz-constants of $\varphi$ and $\eta_{\theta}$, the norm of $\varphi$, $F$, and $\tau$ such that for $\dens > C$ we obtain
	\[
	\sup_{(x, y, t) \in \R^2 \times \R^2 \times (0, T)} u(x, t) - v(y, t) -|x-y| - \tfrac{\gamma}{T-t} < 1,
	\]
	for some constant $\gamma > 0$. Following \cite[Theorem~5.1.]{Crandall1992}, for $R \ge 1$ we define radially symmetric functions $\beta_R(x) \coloneqq \tilde{\beta}_R(|x|)$ with $\tilde{\beta}\in C^\infty([0, \infty))$ and
	\begin{equation}\label{eq:beta}
	\tilde{\beta}_R(t) = 0, \text{ if } t \le R, \text{ and } \tilde{\beta}_R(t) = 2L(t-2R)+1, \text{ if } t \ge 2R,
	\end{equation}
	and $\tilde{\beta}_R$ has to satisfy  $|D\tilde{\beta}_R(x)| + |D^2\tilde{\beta}_R(x)| \le C$ for all $t \in (0, \infty)$ where $C$ is a constant independent of $R$. Moreover, let us define
	\[
	\Phi(x, y, t) \coloneqq u(x, t) - v(y, t) - (1+|x-y|^2)^{\frac{1}{2}} - (\beta_R(x) + \beta_R(y)) - \tfrac{\gamma}{T-t}
	\]
	Due to \eqref{eq:linear_growth} and \eqref{eq:beta}, it holds that $|x|,|y|\to\infty$ or $t \to T$, then $\Phi(x, y, t) \to -\infty$. Indeed, we can compute for $(x, y, t) \in \R^2\times \R^2 \times (0, T)$ with $|x|, |y| \ge 2R$ that
	\begin{align*}
	\Phi(x, y, t) \le L(1+|x|+|y|) -2L(|x|+|y|) - \tfrac{\gamma}{T-t} \le L - |x| - |y| -\tfrac{\gamma}{T-t}.
	\end{align*}
	Hence, there is a triple $(\hat x, \hat y, \hat t)$ at which $\Phi$ assumes its maximum. Now, if the growth estimate fails, then for $R$ big enough, we have $\Phi(\hat x, \hat y, \hat t) > 0$. We will show that even when $\Phi(\hat x, \hat y, \hat t) > 0$ the growth estimate holds. Note, that this condition on $\Phi$ gives us a lower estimate on the difference of $u$ and $v$, i.~e.,
	\[
	|\hat x-\hat y| + \tfrac{\gamma}{T-\hat t} \le u(\hat x, \hat t)- v(\hat y, \hat t).
	\]
	Furthermore, as $\hat{t}$ cannot be zero by the boundary condition, we can apply the Jensen-Ishii-Lemma \cite{Crandall1992}. Moreover, we will write $u$ for $u(\hat x, \hat t)$ and $v$ for $v(\hat y, \hat t)$ to simplify notation. Due to the Jensen-Ishii-Lemma there are
	\begin{align*}
	(a, p + D\beta_R(\hat{x}), X + D^2\beta_R(\hat{x})) &\in \mathcal{P}^{2, +} u(\hat{x}, \hat{t}),\\
	(b, p - D\beta_R(\hat{y}), -X - D^2\beta_R(\hat{y})) &\in \mathcal{P}^{2, -} v(\hat{y}, \hat{t}),
	\end{align*}
	with $a = b+ \frac{\gamma}{(T-\hat{t})^2}$, $p = \frac{\hat{x}-\hat{y}}{1+ |\hat{x}-\hat{y}|^2}$, and $$X = \frac{1}{1+|\hat{x}-\hat{y}|^2} \operatorname{Id} - 2 \frac{\hat{x}-\hat{y}}{1+|\hat{x}-\hat{y}|^2} \otimes \frac{\hat{x}-\hat{y}}{1+|\hat{x}-\hat{y}|^2}.$$
	This implies that one can find $\mu \in \mathcal{S}(a + \dens u)$, and $\nu \in \mathcal{S}(b + \dens v)$ such that
	\begin{align}
	F_*(u, a, p + D\beta_R(\hat{x}), X + D^2\beta_R(\hat{x})) &\le \mu G(\hat{x}, u, p + D\beta_R(\hat{x})) \label{eq:comp_rn_subsol}, \\
	F^*(v, b, p - D\beta_R(\hat{y}), -X - D^2\beta_R(\hat{y})) &\ge \nu G(\hat y, v, p - D\beta_R(\hat{y})) \label{eq:comp_rn_supersol}.
	\end{align}
	Subtracting \eqref{eq:comp_rn_subsol} from \eqref{eq:comp_rn_supersol} shows that
	\begin{align*}
	0 \le&~F^*(v, b, p - D\beta_R(\hat{y}), -X - D^2\beta_R(\hat{y}))
	- F_*(u, a, p + D\beta_R(\hat{x}), X + D^2\beta_R(\hat{x})) \\
	&- \nu\left( G(\hat y, v, p - D\beta_R(\hat{y}))
	- G(\hat x, u, p + D\beta_R(\hat{x}))\right)
	+ (\mu - \nu) G(\hat x, u, p + D\beta_R(\hat{x})).
	\end{align*}
	As $G\ge 0$ and $\mu-\nu \le 0$, we can estimate the last term by zero. Moreover, we can apply lemma \ref{lem:estimates} and obtain
	\begin{align*}
	0 \le &- \tfrac{\gamma}{T^2} + \dens(v-u) \\
	&+ \operatorname{MC}^*(p - D\beta_R(\hat{y}), -X - D^2\beta_R(\hat{y})) - \operatorname{MC}_*(p + D\beta_R(\hat{x}), X + D^2\beta_R(\hat{x})) \\
	&+ (\tau + F)|D\beta_R(\hat{x})+D\beta_R(\hat y)|
	+ L_{\varphi}|p+D\beta_R(\hat{x})| |\hat x - \hat y| \\
	&+||\varphi||_\infty  \left|\eta_{\theta}(\tfrac{v}{|p-D\beta_R(\hat{y})|})|p-D\beta_R(\hat{y})| - \eta_{\theta}(\tfrac{u}{|p+D\beta_R(\hat{x})|})|p+D\beta_R(\hat{x})|\right|
	\end{align*}
	As $p, X, D\beta_R, D^2\beta_R$ are bounded independently of $R$, we find a constant $C > 0$ depending on all data but $\dens$ such that (the last term is bounded by $||\varphi||_\infty \max\{|p+D\beta_R(\hat x)|, |p-D\beta_R(\hat y)|\}$)
	\begin{align*}
	0 &\le \dens (v-u) + C + C \left|v - u\right| + C |\hat x - \hat y| \\
	&\le \dens (v-u) + C + C \left|v - u\right| + C (u-v- \tfrac{\gamma}{T-\hat t})
	\end{align*}
	Hence, we obtain
	\[
	(\dens - 2C)(u-v) \le C
	\]
	and by choosing $\dens > 3C$ we see that $u-v$ is bounded independently of $R\ge 1$ by $1$.
	
	We conclude that for any $(x, y, t) \in \R^2 \times \R^2 \times (0, T)$ we have
	\[
	\Phi(x, y, t) \le \Phi(\hat x, \hat y, \hat t) \le 1,
	\]
	which shows by sending $R\to \infty$, the growth estimate.
	\[
	u(x, t) - v(x, t) - |x-y|-\tfrac{\gamma}{T-t} \le \Phi(x, y, t) \le 1.
	\]
	\textit{Step 2: Comparison Principle}\\
	Assume that $u$ will not stay below $v$, then there is some $\delta > 0$ with
	\begin{equation}\label{eq:assump}
	\sup_{(x, t) \in \R^2 \times (0, T)} u(x, t) - v(x, t) = \delta.
	\end{equation}
	Again, we use the variable doubling technique and introduce the following quantity,
	\[
	M_{\alpha, \gamma, \epsilon} \coloneqq \sup_{(x, y, t) \in \R^2 \times \R^2 \times (0, T)} u(x, t) - v(y, t) - \alpha |x-y|^4 - \epsilon (|x|^2 + |y|^2) - \tfrac{\gamma}{T-t}
	\]
	Due to the growth estimate $M_{\alpha, \gamma, \epsilon}$ is uniformly bounded and it is easy to see that if $\gamma, \epsilon$ are small enough then $M_{\alpha, \gamma, \epsilon}\ge \tfrac{\delta}{2}$. Again using the growth estimate, we obtain uniform bounds on $\alpha |x-y|^4$ and $\epsilon(|x|^2+|y|^2)$, i.~e.,
	\[
	\alpha |x-y|^4 + \epsilon(|x|^2+|y|^2) \le u(x, t) - v(y, t) - \tfrac{\gamma}{T-t} \le |x-y| + 1.
	\]
	Using Young's inequality, $|x-y| \le \frac{3}{4}\alpha^{-\frac{1}{3}} + \frac{\alpha}{4} |x-y|^4$, we see that
	\begin{equation}\label{eq:bound_on_diff_x}
	\alpha |x-y|^4 + \epsilon(|x|^2+|y|^2) \le 2
	\end{equation}
	for $\alpha$ bigger than some geometric quantity. This proves that the maximum of $M_{\alpha, \gamma, \epsilon}$ is achieved in a compact subset of $\R^2\times\R^2\times[0, T)$ at some triple $(\hat x, \hat y, \hat t)$. Again, the initial condition reveals that $\hat t$ cannot be zero and we can again apply the Jensen-Ishii-Lemma and obtain
	\begin{align*}
	(a, p+2\epsilon \hat{x}, X + 2\epsilon\operatorname{Id}) &\in \mathcal{P}^{2, +} u(\hat{x}, \hat{t}), \\
	(b, p-2\epsilon \hat{y}, Y-2\epsilon\operatorname{Id}) &\in \mathcal{P}^{2, -} v(\hat{y}, \hat{t}) ,
	\end{align*}
	with $a-b = \tfrac{\gamma}{(T-\hat{t})^2}$, $p \coloneqq 4\alpha |\hat x-\hat y|^2 (\hat x-\hat y)$, and
	\[
	-4||Z|| \left(\begin{array}{cc}\operatorname{Id} & 0 \\ 0 & \operatorname{Id}\end{array}\right) \le \left(\begin{array}{cc}X & 0 \\ 0 & Y\end{array}\right)\le \left(\begin{array}{cc}Z+\frac{1}{2||Z||}Z^2 & -(Z+\frac{1}{2||Z||}Z^2) \\ -(Z+\frac{1}{2||Z||}Z^2 ) & Z+\frac{1}{2||Z||}Z^2 \end{array}\right),\]
	with $Z \coloneqq 4 \alpha |\hat x-\hat y|^2 \operatorname{Id} + 8\alpha (\hat x-\hat y) \otimes (\hat x-\hat y)$. As $u$ is a subsolution and $v$ is a supersolution, we can find $\mu \in \mathcal{S}(a + \dens u)$ and $\nu \in \mathcal{S}(b + \dens v)$ such that
	\begin{align}
	F_*(u, a, p+2\epsilon\hat x, X+ 2\epsilon\operatorname{Id}) - \mu G(\hat x, u, p + 2\epsilon\hat x) \le 0,\label{eq:comp_rn_proof_sub} \\
	F^*(v, b, p - 2\epsilon\hat y, Y- 2\epsilon\operatorname{Id}) - \nu G(\hat y, v, p- 2\epsilon\hat y) \ge 0,\label{eq:comp_rn_proof_super}.
	\end{align}
	By subtracting \eqref{eq:comp_rn_proof_sub} from \eqref{eq:comp_rn_proof_super}, we obtain
	\begin{align*}
	0 \le&~ F^*(v, b, p - 2\epsilon \hat y, Y- 2\epsilon\operatorname{Id}) - F_*(u, a, p + 2\epsilon \hat x, X + 2\epsilon\operatorname{Id}) \\
	&- \nu\left( G(\hat y, v, p -  2\epsilon \hat{y})
	- G(\hat x, u, p + 2\epsilon \hat{x})\right)
	+ (\mu - \nu) G(\hat x, u, p +2\epsilon \hat{x}).
	\end{align*}
	Again, the last term can be estimated from above by zero and we apply lemma \ref{lem:estimates} to the first two differences. This leads to
	\begin{equation}\label{eq:inequality}
	\begin{aligned}
	0\le&~-\tfrac{\gamma}{T^2}+\dens(v-u)+ \operatorname{MC}^*(p - 2\epsilon \hat y, Y) - \operatorname{MC}_*(p + 2\epsilon \hat x, X) + 2\epsilon |\hat x + \hat y|F\\
	&+ 2\tau\epsilon|\hat x+\hat y|+ L_\varphi |p+2\epsilon\hat{x}| |\hat x-\hat y| \\
	&+ ||\varphi||_\infty \left|\eta_{\theta}(\tfrac{v}{|p-2\epsilon\hat{y}|})|p-2\epsilon\hat{y}| - \eta_{\theta}(\tfrac{u}{|p+2\epsilon\hat{x}|})|p+2\epsilon\hat{x}|\right|.
	\end{aligned}
	\end{equation}
	In the next step, we want to take the limit inferior for $\epsilon \to 0$, note that in this case both $\epsilon \hat x, \epsilon \hat y \to 0$ and $|\hat x - \hat y|$ remains uniformly bounded independently of $\epsilon$. Hence, it remains to investiage the limit of the curvature terms and the last term. Hence, let us start with the following estimate
	\begin{align*}
	&\left|\eta_{\theta}(\tfrac{v}{|p-2\epsilon\hat{y}|})|p-2\epsilon\hat{y}| - \eta_{\theta}(\tfrac{u}{|p+2\epsilon\hat{y}|})|p+2\epsilon\hat{x}|\right| \\
	&~\le \left|\eta_{\theta}(\tfrac{v}{|p-2\epsilon\hat{y}|})|p-2\epsilon\hat{y}| - \eta_{\theta}(\tfrac{u}{|p-2\epsilon\hat{y}|})|p-2\epsilon\hat{y}|\right| \\
	&\hspace{2em}+ \left| \eta_{\theta}(\tfrac{u}{|p-2\epsilon\hat{y}|})|p-2\epsilon\hat{y}|
	-\eta_{\theta}(\tfrac{u}{|p+2\epsilon\hat{x}|})|p+2\epsilon\hat{x}|\right| \\
	&~\le L_{\eta_{\theta}} |u-v| + \left| \eta_{\theta}(\tfrac{u}{|p-2\epsilon\hat{y}|})|p-2\epsilon\hat{y}|
	-\eta_{\theta}(\tfrac{u}{|p+2\epsilon\hat{x}|})|p-2\epsilon\hat{y}|\right|\\
	&\hspace{2em}+\left|\eta_{\theta}(\tfrac{u}{|p+2\epsilon\hat{x}|})|p-2\epsilon\hat{y}|
	-\eta_{\theta}(\tfrac{u}{|p+2\epsilon\hat{x}|})|p+2\epsilon\hat{x}|\right| \\
	&~\le  L_{\eta_{\theta}} |u-v| + |p-2\epsilon\hat{y}|\left| \eta_{\theta}(\tfrac{u}{|p-2\epsilon\hat{y}|})
	-\eta_{\theta}(\tfrac{u}{|p+2\epsilon\hat{x}|})\right| + 2\epsilon|\hat x + \hat y|.
	\end{align*}
	Note, that the second term vanishes if $|u|\ge \theta \max\{|p-2\epsilon \hat y|, |p+2\epsilon \hat x|\}$. On the other hand if $u$ were smaller then this quantity, we can proceed with the following estimate
	\begin{align*}
	|p-2\epsilon\hat{y}|\left| \eta_{\theta}(\tfrac{u}{|p-2\epsilon\hat{y}|})
	-\eta_{\theta}(\tfrac{u}{|p+2\epsilon\hat{x}|})\right| &\le L_{\eta_\theta} |u| \left|1- \tfrac{ |p-2\epsilon\hat{y}|}{|p+2\epsilon\hat{x}|}\right| \\
	&\le L_{\eta_\theta}\theta \max\{|p-2\epsilon \hat y|, |p+2\epsilon \hat x|\} \left|1- \tfrac{ |p-2\epsilon\hat{y}|}{|p+2\epsilon\hat{x}|}\right|
	\end{align*}
	Note that this term vanishes for $\epsilon \to 0$ as $|p|$ remains bounded independently of $\epsilon$, hence we can conclude that
	\[
	\left|\eta_{\theta}(\tfrac{v}{|p-2\epsilon\hat{y}|})|p-2\epsilon\hat{y}| - \eta_{\theta}(\tfrac{u}{|p+2\epsilon\hat{y}|})|p+2\epsilon\hat{x}|\right| \le L_{\eta_\theta} |u-v| + {\rm o}(1) \text{ as } \epsilon \to 0
	\]
	Reintroducing this information in \eqref{eq:inequality}, we obtain for $\epsilon \to 0$ that
	\[
	0 \le -\tfrac{\gamma}{T^2}+(\dens-L_{\eta_\theta})(v-u)+ \operatorname{MC}^*(p - 2\epsilon \hat y, Y) - \operatorname{MC}_*(p + 2\epsilon \hat x, X)+ L_\varphi |p| |\hat x-\hat y| + {\rm o}(1).
	\]
	While taking the limit inferior for $\epsilon \to 0$, the difference of the curvature termes becomes negative, as either the limit point of $p$ is non-zero, which leaves us with $\operatorname{MC}(p, Y) - \operatorname{MC}(p, X)$ and then the degenerate ellipticity of the curvature operator applies. Moreover, if the $p$ were zero, then by definition $X=Y=0$ and the operator vanishes. Further choosing $\dens \ge L_{\eta_{\theta}} +2$, we obtain
	\[
	0 \le -\tfrac{\gamma}{T^2}-\delta+ \liminf_{\epsilon \to 0} L_\varphi |p| |\hat x-\hat y|.
	\]
	Note that $|p||\hat x-\hat y|$ is nothing but $4\alpha|\hat x-\hat y|^4$ and therefore we are done if we can prove that
	\begin{equation}\label{eq:fourth_power}
	\liminf_{\alpha \to \infty}\liminf_{\gamma\to 0}\liminf_{\epsilon \to 0} \alpha |\hat x - \hat y|^4 = 0,
	\end{equation}
	as this would lead to a contradiction.
	
	To see that equation \eqref{eq:fourth_power} is true, we follow \cite{Giga1991}. First, we define
	\[
	\delta(r) \coloneqq \sup_{(x, y, t) \in \R^2 \in \R^2 \times (0, T)} \{ u(x, t) - v(y, t) \;|\; |x-y|<r \}
	\]
	and set $\delta_0 \coloneqq \liminf_{r \to 0} \delta(r)$. Note that $\delta_0 \ge \delta > 0$. For any $r > 0$, take a maximizing sequence $(x^r_n, y^r_n, t^r_n)_{n\in \N} \subset \R^2 \times \R^2 \times (0, T)$ with $|x^r_n -y^r_n|<r$ for all $n$, i.~e., $u(x^r_n, t^r_n) - v(y^r_n, t^r_n) \to \delta(r)$ for $n\to \infty$. Then it holds,
	\[
	M_{\alpha, \gamma, \epsilon} \ge u(x^r_n, t^r_n) - v(y^r_n, t^r_n) - \alpha r^4 - \epsilon (|x^r_n|^2 + |y^r_n|^2) - \tfrac{\gamma}{T-t^r_n}
	\]
	Hence,
	\[
	\liminf_{\gamma \to 0} \liminf_{\epsilon \to 0} M_{\alpha, \gamma, \epsilon}\ge u(x^r_n, t^r_n) - v(y^r_n, t^r_n) - \alpha r^4.
	\]
	As the left handside is independent of $n$, we can pass to the limit revealing that $\liminf_{\gamma \to 0} \liminf_{\epsilon \to 0} M_{\alpha, \gamma, \epsilon} \ge \delta(r) - \alpha r^4$ and finally, we can send $r\to 0$ and see the following lower bound,
	\begin{equation}\label{eq:M_lower_bound}
	\liminf_{\gamma \to 0} \liminf_{\epsilon \to 0} M_{\alpha, \gamma, \epsilon} \ge \delta_0.
	\end{equation}
	On the other hand as $|\hat x - \hat y|\to 0$ for $\alpha \to \infty$, there exists some decreasing continuous function $\omega : [0, \infty) \to [0, \infty)$ with $|\hat x - \hat y| = \omega(\alpha^{-1})$ and therefore
	\begin{align*}
	M_{\alpha, \gamma, \epsilon} &= u(\hat x, \hat t) - v(\hat y, \hat t) - \alpha |\hat x - \hat y|^4 - \epsilon (|\hat x|^2 + |\hat y|^2) - \tfrac{\gamma}{T-\hat t} \\
	&\le \delta(\omega(\alpha^{-1})) - \alpha |\hat x - \hat y|^4.
	\end{align*}
	Combining this estimate with \eqref{eq:M_lower_bound}, we obtain
	\[
	\liminf_{\gamma \to 0}\liminf_{\epsilon \to 0} \alpha |\hat x - \hat y|^4 \le \delta(\omega(\alpha^{-1})) - \liminf_{r\to 0} \delta(r).
	\]
	As $\liminf_{\alpha \to \infty} \delta(\omega(\alpha^{-1})) = \liminf_{r\to 0} \delta(r)$, we can take the limit inferior as $\alpha \to \infty$ and see that \eqref{eq:fourth_power} holds.
\end{proof}

Recall, that we have just proven that \eqref{eq:main_levelset} satisfies a comparison principle. Due to Perron's method (see \cite{CourteDondl} for a proof of Perron's method for our type of equation), it is enough to show that there is a viscosity subsolution and a viscosity supersolution that satisfy the boundary condition in a strong sense in order to obtain a unique viscosity solution.

\begin{proposition}[Existence]
	For all $F, \theta > 0$, there exists a unique viscosity solution $u:\R^2 \times [0, +\infty) \to \R$ to \eqref{eq:main_levelset} with $\mathcal{F}$ as in \eqref{eq:dissipation} and it holds
	\[
	u(x, y, 0) = -y.
	\]
	Hence, this solution describes the evolution of an interface with initial zero level-set $\{ y = 0 \}$.
\end{proposition}
\begin{proof}
	Define $\underline{u}(x, y, t) \coloneqq -y$ and $\overline{u}(x, y, t) \coloneqq -y + \Lambda t$ with $\Lambda > 0$ big enough. Note that for each final time $T < \infty$ any sub- or supersolution that lies between $\underline{u}$ and $\overline{u}$ has the required growth \eqref{eq:linear_growth}, as
	\[
	|\underline{u}(x, y, t) - \overline{u}(x, y, t)| \le \Lambda T.
	\]
	Hence, for every $T > 0$ there is a unique viscosity solution to \eqref{eq:main_levelset} on $\R^2 \times [0, T)$. By uniqueness on each finite time interval, we can easily construct a unique continuous solution on $\R^2 \times [0, \infty)$.
\end{proof}

\section{Scaling Lower Bound}\label{sec:lower_bound}

Under the assumptions of theorem \ref{thm:existence_of_subsupersol} the existence of a stationary supersolution has already been proven in \cite{Dondl_11b}. However, the construction used in the article is not optimal and tracking all the constants leads to a lower bound that scales like $\dens$. We are improving on their construction, specifically, we are not gluing together flat solutions, but directly connect two solutions that lie inside obstacles with an arc of a circle. As in \cite{Dondl_11b} we have to track the derivative of the connecting circle arc in order to guarantee that the mean curvature at the point where we connect the functions remains a negative measure.

\begin{lemma}[Connecting with circles]\label{lem:connect}
	A point $(x_1, y_1) \in \R^2$ can be connected to a point $(x_2, y_2) \in \R^2$ with a function $u$ that parametrises an arc of a circle of (negative) curvature $\kappa \ge 0$ so that the derivatives at the leftmost point is less than $\alpha \ge 0$ and at the right point is greater than $-\alpha$
	\[
	\frac{\alpha}{\sqrt{1+\alpha^2}} \ge \frac{\kappa \bar{x}}{2} + \kappa \bar{y}\sqrt{\frac{\kappa^{-2}}{\bar{x}^2+\bar{y}^2} - \frac{1}{4}},
	\]
	and
	\[
	\kappa \le \frac{2\bar{x}}{\bar{x}^2 +\bar{y}^2},
	\]
	where $\bar{x} = |x_2 - x_1|, \bar{y} = |y_2 - y_1|$.
\end{lemma}

\begin{proof}
	Assume wlog that $x_1 \le x_2$ and $y_1 \le y_2$ and define $\overline{x} \coloneqq x_2 - x_1$ and $\overline{y} \coloneqq y_2 - y_1$. It is then enough to prove the result in the case where $(x_1, y_1) = (0, 0)$ and $(x_2, y_2) = (\overline{x}, \overline{y})$.
	
	Define $u(x) \coloneqq \sqrt{\kappa^{-2} - (l-x)^2} - \sqrt{\kappa^{-2} - l^2}$ where
	\[
	l \coloneqq \frac{\overline{x}}{2} + \overline{y} \sqrt{\frac{\kappa^{-2}}{\overline{x}^2 + \overline{y}^2} - \frac{1}{4}}.
	\]
	Note that $u$ and $l$ are well-defined as $\kappa \le \frac{2\bar{x}}{\bar{x}^2 +\bar{y}^2}$. With this choice of $l$ it holds that $u(0) = 0$ and $u(\overline{x}) = \overline{y}$.
	
	Moreover, note that
	\[
	u'(x) = \frac{l-x}{\sqrt{\kappa^{-2} - (l-x)^2}}
	\]
	and that $u'(0) \le \alpha$ and $u'(\overline{x}) \ge -\alpha$ if $\kappa \le \frac{\alpha}{\sqrt{1+\alpha^2}} \frac{1}{l}$.
	
	Finally, note that $$\div\left( \frac{\nabla u}{\sqrt{1+|\nabla u|^2}}\right) = \frac{\mathrm{d}}{\mathrm{d}x} \frac{u'}{\sqrt{1+u'}} = -\kappa.$$
\end{proof}

\begin{lemma}[Local solution inside an obstacle]\label{lem:loc_inside_obstacle}
	Let $r > 0, \overline{f} \ge F \ge 0$ then there exists a solution $w \in C^1((-r, r))$ to
	\[
	-\div\left(\frac{\nabla w}{\sqrt{1+|\nabla w|^2}}\right) + \overline{f} \ge F \text{ in } (-r, r)
	\]
	with $w(-r) = w(r) = 0$ and $\div\left(\frac{\nabla w}{\sqrt{1+|\nabla w|^2}}\right) = F_{\rm in} \ge 0$ if
	\[
	F_{\rm in} \le \min\{\overline{f} - F, r^{-1}\}.
	\]
	
	Moreover, if the solution exists, then it holds that
	\[
	w'(r) = -w'(-r) = \frac{r}{\sqrt{F_{\rm in}^{-2} - r^2}}.
	\]
\end{lemma}
\begin{proof}
	Define $w(x) \coloneqq -\sqrt{F_{\rm in}^{-2} - x^2} + \sqrt{F_{\rm in}^{-2} - r^2}$ which is well-defined if $r \le \frac{1}{F_{\rm in}}$. Additionally, it holds
	\[
	w'(x) = \frac{x}{\sqrt{F_{\rm in}^{-2} - x^2}}
	\]
	and
	\[
	-\div\left(\frac{\nabla w}{\sqrt{1+|\nabla w|^2}}\right) + \overline{f} = -F_{\rm in} + \overline{f} \ge F.
	\]
	which proves the result.
\end{proof}

\begin{theorem}[Lower bound]\label{thm:lower_bound}
	There is a constant $c(r_0, f) > 0$ depending only on $r_0$ and $\overline{f}$ such that for each $$F \le  \sup \mathcal{F}(0) + \min\left\{ c(r_0, f) \sqrt{\dens}, \frac{f}{2} \right\}$$ the equation
	\[
	\mathcal{F}\left(v[u](x, t)\right) -\div\left(\tfrac{\nabla u(\cdot, t)}{\sqrt{1+|\nabla u(\cdot, t)|^2}}\right)(x) + \varphi(x, u(x, t), \omega)\partial R\left(v[u](x, t)\right) \ni F \text{ in } \R \times [0, \infty)
	\]
	where $v[u](x, t) = \frac{u_t(x, t)}{\sqrt{1+|\nabla u(x, t)|^2}}$ has almost surely a stationary, i.~e., $u_t = 0$, viscosity super-solution.
\end{theorem}

\begin{proof}
	As we want to construct a stationary super-solution and we have the freedom to choose any value in the set-valued terms $\mathcal{F}(0)$ and $\partial R(0)$ (see \cite{CourteDondl}), it is enough to find a super-solution to the equation
	\[
	-\div\left(\frac{\nabla u(\cdot, t)}{\sqrt{1+|\nabla u(\cdot, t)|^2}}\right)(x) + \varphi(x, u(x, t), \omega) = \tilde{F} \text{ in } \R \times [0, \infty),
	\]
	where $\tilde{F} = F - \sup \mathcal{F}(0)$.
	
	Let $d > 0$, $h>0$ and $\overline{f} \ge 0$ and define $l = C_0 h^{-1} + 2 r_{1}$ with
	\[
	C_0 = \left(\frac{-\log(1-p_c)}{\dens}\right)^{\frac{1}{n}},
	\]
	where $n = 1$ and $p_c = 1- (2n+2)^{-2}$. If we consider for each $k \in \mathbb{Z}$ and $j \in \N$ the cubes
	\begin{align*}
	\tilde{Q}_k &\coloneqq [k(l+d) + r_1, k(l+d)+l-r_1],\\
	\bar{Q}_k &\coloneqq [k(l+d), k(l+d)+l],\\
	\tilde{Q}_{k, j} &\coloneqq \tilde{Q}_k \times [(j-1)h+r_1, jh+r_1],
	\end{align*}
	then there is almost surely a random function $L:\mathbb{Z} \to \N$ with Lipschitz constant $1$ such that for all $k \in \mathbb{Z}$ there is $i \in \N$ such that
	\[
	(x_i, y_i) \in \tilde{Q}_{k, L(k)}.
	\]
	The existence of such an $L$ follows by the Lipschitz Percolation result from Dirr-Dondl-Grimett-Holroyd-Scheutzow, see also \cite[Proposition~2.8]{Dondl_10a}.
	
	This means that we have to connect an obstacle that lies inside $\tilde{Q}_{k, L(k)}$ with an obstacle that lies inside $\tilde{Q}_{k-1, L(k-1)}$ and an obstacle that lies inside $\tilde{Q}_{k+1, L(k+1)}$. Due to the symmetrie of the situation, it is enought to connect an obstacle inside $\tilde{Q}_{0, 0}$ with an obstacle that lies somewhere in $\tilde{Q}_{1,0} \cup \tilde{Q}_{1, 1}$.
	
	Let $r \le r_0$ then we start by constructing the local solution inside an obstacle, i.~e., we apply lemma \ref{lem:loc_inside_obstacle} for any selected obstacle, and we obtain a local solution inside the obstacle with curvature $F_{\rm in} \le \min\{ \overline{f}-\tilde{F}, r^{-1} \}$ and outward derivatives of intensity $\alpha\coloneqq\frac{r}{\sqrt{F_{\rm in}^{-2} - r^2}}$.
	
	Note, now that the distance $(\bar{x}, \bar{y}) = (|x_i - x_j|, |y_i - y_j|)$ between the selected obstacle $i$ in $\tilde{Q}_{0, 0}$ and the selected obstacle $j$ in $\tilde{Q}_{1,0} \cup \tilde{Q}_{1, 1}$ satisfies
	\begin{align*}
	\overline{x} &\in [d+2r_{1}-2r, d+ 2l - 2r_1-2r],\\
	\overline{y} &\in [0, 2h].
	\end{align*}
	Moreover, we assume that $2h \ge d$ so that $\bar{x} \ge \bar{y}$. Now, we have to connect $(x_i+r, y_i)$ with $(x_j-r, y_j)$ with an arc of a circle as in lemma \ref{lem:connect} such that the curvature is $F_{\rm out} \ge \tilde{F}$.
	
	Furthermore, we assume that $F_{\rm out} \ge \frac{\bar{y}}{\bar{x}^2} \ge \frac{2\bar{y}}{\bar{x}^2+\bar{y}^2}$ so that $\frac{\alpha}{\sqrt{1+\alpha^2}} \ge F_{\rm out} \overline{x}$ implies that $\frac{\alpha}{\sqrt{1+\alpha^2}} \ge \frac{F_{\rm out} \bar{x}}{2} + F_{\rm out} \bar{y}\sqrt{\frac{F_{\rm out}^{-2}}{\bar{x}^2+\bar{y}^2} - \frac{1}{4}}$.
	
	Finally, note that $\frac{\alpha}{\sqrt{1+\alpha^2}} = F_{\rm in} r$ and bring all together to see that if we choose $d=2C_0 h^{-1}$ and $F_{\rm out} = \frac{F_{\rm in} r}{4C_0 h^{-1}}$ we obtain
	\[
	F_{\rm out} = \frac{F_{\rm in} r}{4C_0 h^{-1}} \le \frac{F_{\rm in}r}{d + 2C_0 h^{-1} + 2r_1 - 2r} \le \frac{F_{\rm in}r}{\bar x}
	\]
	and
	\[
	F_{\rm out}=\frac{F_{\rm in} r}{4C_0 h^{-1}} \ge \frac{2h}{(4C_0h^{-1})^2} \ge \frac{\overline{y}}{\overline{x}^2}
	\]
	if $h \le \sqrt{2} (F_{\rm in} r)^{\frac{1}{2}} C_0^{\frac{1}{2}}$, hence in order to maximize $F_{\rm out}$ we choose $h$ so that equality holds. This leads to
	\[
	F_{\rm out} = \frac{\sqrt{2}(F_{\rm in}r)^{\frac{3}{2}}}{4} C_0^{-\frac{1}{2}}.
	\]
	
	Hence, we find a solution to the equation as long as
	\[
	\tilde{F} \le \min\left\{ \frac{\sqrt{2}(F_{\rm in}r)^{\frac{3}{2}}}{4} C_0^{-\frac{1}{2}}, f - F_{\rm in} \right\}.
	\]
	Now choosing $F_{\rm in} = \frac{f}{2}$ and $r = \min\{r_0, 2f^{-1} \}$, we obtain
	\[
	F - \sup \mathcal{F}(0) = \tilde{F} \le \min\left\{ c(r_0, f) \sqrt{\dens}, \frac{f}{2} \right\}.
	\]
\end{proof}

\begin{theorem}[Level-set supersolution]\label{thm:levelset_super}
	Assume that $\theta > 0$ is small enough, then under the assumptions of theorem \ref{thm:lower_bound} there exists almost surely a level-set supersolution to equation \eqref{eq:main_levelset}.
\end{theorem}
\begin{proof}
	Let $v$ be the graphical supersolution from theorem \ref{thm:lower_bound} and define
	\[
	\overline{u}(x, y, t) \coloneqq v(x) - y.
	\]
	First, we note that the zero levelset of $\overline{u}$ is the graph of $v$ and that $\overline{u}_t = 0$, i.~e., $\overline{u}$ is stationary.
	
	Due to the previous theorem, there is some $\beta$ depending on all parameters such that for all $x \in \R \times I$ there is some $\mu \in \mathcal{F}(0)$ such that
	\begin{align*}
	&\sqrt{1+|\nabla v(x)|^2}\mu - \sqrt{1+|\nabla v(x)|^2}\div\left(\tfrac{\nabla v(x)}{\sqrt{1+|\nabla v(x)|^2}}\right) + \sqrt{1+|\nabla v(x)|^2}\varphi(x, v(x)) \\
	&~\ge \sqrt{1+|\nabla v(x)|^2}(F+\beta).
	\end{align*}
	As $|\nabla \overline{u}(x, y, t)| = \sqrt{1+|\nabla v(x)|^2}$ for all $(x,y,t) \in \R^2 \times I$, we can rewrite the previous equation in the following way
	\[
	|\nabla \overline{u}|\mu - |\nabla \overline{u}| \div(\tfrac{\nabla \overline{u}}{|\nabla \overline{u}|}) + |\nabla \overline{u}|\varphi(x, v(x)) \ge |\nabla \overline{u}|(F+\beta)
	\]
	Now, choosing $\theta$ small enough (depending on all parameters) we see that
	\[
	|\nabla \overline{u}|\mu - |\nabla \overline{u}| \div(\tfrac{\nabla \overline{u}}{|\nabla \overline{u}|}) + |\nabla \overline{u}|\varphi(x, \overline{u})\eta_{\theta}(\tfrac{\overline{u}}{|\nabla \overline{u}|}) \ge |\nabla u|F.
	\]
	Now, noting that $\mu \in \mathcal{F}(\tfrac{\overline{u}_t}{|\nabla \overline{u}|}) = \mathcal{F}(0)$ shows that $\overline{u}$ is a level-set supersolution to \eqref{eq:main_levelset}.
\end{proof}

\section{Scaling Upper Bound}\label{sec:upper_bound}

In order to show the existence of a propagating subsolution to equation \eqref{eq:main}, it is enough to find an unbounded path of positive width that start at $\{y=0\}$. Inside this path, we can then construct a propagating subsolution. However, such a propagating subsolution is only ballistic if we can guarantee that we find distinct paths with a positive density. We can use a percolation result for stacked Lipschitz surfaces from A.E. Holroyd and  J.B. Martin \cite[Theorem~2.1]{holroyd2014} which we restate here for the convenience of the reader for $d=2$.
\begin{theorem}[Stacked Lipschitz surfaces {\cite[Theorem~2.1]{holroyd2014}}, $d=2$]\label{thm:stacked_lipschitz_surface}
	Consider site percolation on $\mathbb{Z}^2$, i.~e., a site $(k, l) \in \mathbb{Z}^2$ is open with a probability $p$ and closed with probability $1-p$ (independently of other sites). If $p\ge p_c$, where $p_c < 1$ is deterministic, then a.s. there are (random) functions $L_n : \mathbb{Z} \to \mathbb{Z}$, $n\in \mathbb{Z}$, with the following properties:
	\begin{enumerate}
		\item[(L1)] The site $(k, L_n(k))$ is open for all $k\in \mathbb{Z}$, $n \in \mathbb{N}$.
		\item[(L2)] For each $n$, $L_n$ is $1$-Lipschitz, i.~e., $|L_n(k) - L_n(l)| \le 1$ whenever $|k-l|= 1$.
		\item[(L3)] $L_n(k) > 2n$ for all $k\in \mathbb{Z}$, $n \in \mathbb{Z}$.
		\item[(L4)] $L_{n-1}(k) < L_n(k)$ for all $k\in \mathbb{Z}$, $n \in \mathbb{Z}$.
		\item[(L5)] $(L_{n+m}(k+l) - 2m)_{n, k}$ and $(L_n(k))_{n, k}$ are equal in law for all $l \in \mathbb{Z}$ and $m\in \mathbb{Z}$.
	\end{enumerate}
\end{theorem}

\begin{remark}\label{rk:applying_stacked_lipschitz_surfaces}
	Note that \textit{(L5)} implies that the mean distance between two adjacent Lipschitz surfaces is two, as
	\[
	E[L_{n+1}(k) - L_n(k)] = E[2 + L_n(k) - L_n(k)] = 2.
	\]
	This shows that the Lipschitz-surfaces occur with a positive density. As the Poisson Point Process, which we use to generate the obstacles, see definition \ref{def:obstacles}, is ergodic, we can apply a rotation to the whole room to find Lipschitz surfaces that are not horizontal but vertical. 	
\end{remark}

We can apply this theorem and the remark to construct different paths that do not intersect an occur with a positive density. We start by dividing the domain in cubes of side-length $h$, i.~e., define
\[
Q_{k, l} = [hk, h(k+1)] \times [hl, h(l+1)].
\]

A path $P$ of width $h$ is the union of a subset of such cubes such that there is a function $p: \mathbb{N} \to \mathbb{N}\times \mathbb{Z}$ with
\[
Q_{k, l} \in P \iff (k, l) \in p(\mathbb{N})
\]
and $p$ is injective, it holds $p(k+1) \in p(k) + \left\{ (-1, 0), (1, 0), (0, 1) \right\}$, and $\lim_{k\to\infty} p(k)_2 = +\infty$.

\begin{lemma}[Existence of paths with density]\label{lem:paths_with_density}
	Assume that $2r_1 < (-\frac{1}{5}\log(p_c))^{\frac{1}{2}} \dens^{-\frac{1}{2}}$, then there exists almost surely a family of functions $p_n:\mathbb{N}\to \mathbb{N}\times \mathbb{Z}$ such that $p_n$ describe a path of width $h =\frac{1}{2} (-\frac{1}{5}\log(p_c))^{\frac{1}{2}}\dens^{-\frac{1}{2}}$ with
	\[
	Q_{p_n(k)} \cap \bigcup B_{r_1}(x_i, y_i) = \emptyset,
	\]
	i.~e., the path does not intersect any obstacle.
\end{lemma}
\begin{proof}
	This follows immediately by theorem \ref{thm:stacked_lipschitz_surface} and remark \ref{rk:applying_stacked_lipschitz_surfaces} considering a cube $Q_{k, l}$ to be open if
	\[
	Q_{k, l} \cup Q_{k+1, l} \cup Q_{k-1, l} \cup Q_{k, l+1} \cup Q_{k, l-1} \text{ contain no obstacle centers }.
	\]
	The probability of this event is $e^{-5\dens h^2}$ and as this probability has to be greater then $p_c$, we see that $h <  (-\frac{1}{5}\log(p_c))^{\frac{1}{2}}\dens^{-\frac{1}{2}}$. We can construct from a Lipschitz-surface a path of width $\tilde{h}$, if $\tilde{h} \le \frac{h_1}{2}$. To guarantee that such a path does not intersect any obstacles $h-\tilde{h}$ has to be greater then $r_1$. Hence, by choosing $\tilde{h} = \frac{h}{2}$ all these conditions are satisfied.
	
	Finally, the different paths shall not intersect. To achieve this, we ignore all Lipschitz-surfaces $L_n$ with $n$ even. This still leads to paths that occur with positive density.
\end{proof}

\begin{lemma}[Local Propagating Subsolution]\label{lem:local_propagating_subsolution}
	Let $P$ be a path of width $h$, then there exists a propagating subsolution $\Gamma(t)$, $\Gamma(0)=\{y=0\}$, such that
	\[
	\mathcal{F}(v_n(\xi)) \le \kappa(\xi) + F \text{ for all } \xi \in \Gamma(t)
	\]
	and for each compact set $K \subset \R^2$ with $\Gamma(t_0) \subset K$ there is a $t_1 > t_0$ such that $\Gamma(t_1) \not\subset K$,
	if $F > \sup\mathcal{F}(0) + \frac{2}{h}$.
\end{lemma}
\begin{proof}
	We construct the solution in the following way. First we start with an initialisation step, and afterwards we discuss how the solution passes through $Q_{p(1)}$. Then we construct the solution by induction. Depending on whether $p(k+1) - p(k) = (0, 1)$, $p(k+1)-p(k) = (-1, 0)$, or $p(k+1)-p(k) = (1, 0)$ we apply a construction so that the solution passes (in the first case) or rotates (in the second and third case) through a cube and remains (up to translation and rotation) again in the initial position (which is also the position of the interface after initialisation). As it does not matter if we rotate left or right, it is enough to prove the following three steps
	\begin{enumerate}
		\item (Re)-Initialisation (see figure \ref{fig:1}),
		\begin{figure}
			\def\svgwidth{150px}
			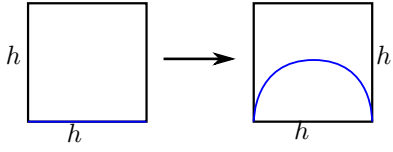
			\caption{Evolution in step (1), due to the force in the normal direction a circle of radius $h/2$ arises.}
			\label{fig:1}
		\end{figure}
		\item Passing through a cube (see figure \ref{fig:2}),
		\begin{figure}
			\def\svgwidth{150px}
			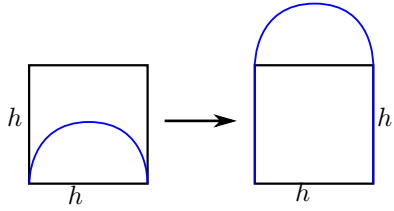
			\caption{Evolution in step (2), this step is used if the next cube is above the current cube and in this case it is enough to just to just gain some height.}
			\label{fig:2}
		\end{figure}
		\item Rotating through a cube (see figure \ref{fig:3}).
		\begin{figure}
			\def\svgwidth{300px}
			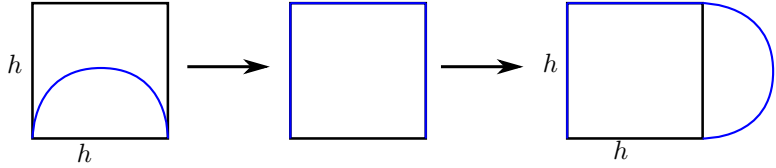
			\caption{Evolution in step (3), note that this is actually only a capped version of step (2) combined with a reinitialisation from step (1).}
			\label{fig:3}
		\end{figure}
	\end{enumerate}
	
	Define $v_0 > 0$ such that  there is some $\mu_0 \in \sup \mathcal{F}(v_0)$ with $\mu_0 \le F - \frac{2}{h}$. Note that this $v_0$ exists by the assumption on $F$. This $v_0$ is more or less the normal velocity on the non-stationary parts of $\Gamma(t)$.	
	\begin{enumerate}
		\item (Re)-Initialisation\newline
		This step assumes that $\Gamma(t_0)$ is flat. After applying a translation (in space and time) and a rotation (in space) we can assume that $Q_{p(0)} = Q \coloneqq [-\frac{h}{2}, \frac{h}{2}]\times[0, h]$ and $\Gamma(0) = \{y=0\}$.
		Define $$\Gamma(t) \cap Q = \left\{
		y = \sqrt{\kappa(t)^{-2} - x^2} - \sqrt{\kappa(t)^{-2} - \left(\tfrac{h}{2}\right)^2}, |x| \le \frac{h}{2}
		\right\}$$
		with $\kappa(t) = \frac{2v_0t}{v_0^2 t^2+\left(\frac{h}{2} \right)^2}$, $t\le \frac{h}{2v_0}$. We only have to check that $\Gamma(t)$ is a subsolution in $Q_{p(0)}$. However, as $\Gamma$ is graphical in this part, we can compute the normal velocity for some $\xi=(x, y(x)) \in \Gamma(t)$ by
		\begin{align*}
		\sup\mathcal{F}\left(v_n(\xi)\right) &= \sup\mathcal{F}\left(\frac{y_t}{\sqrt{1+|\nabla y|^2}}\right) \\
		&= \sup\mathcal{F}\left(\left(\frac{\sqrt{\kappa(t)^{-2} - x^2}}{\sqrt{\kappa(t)^{-2} - \left( \frac{h}{2}\right)^2}} - 1\right) \frac{\kappa'(t)}{\kappa(t)^2}\right) \\
		&\le\sup\mathcal{F}\left(\left(\frac{1}{\sqrt{1 - \kappa(t)^2\left( \frac{h}{2}\right)^2}} - 1\right) \frac{\kappa'(t)}{\kappa(t)^2}\right) \\
		&=\sup\mathcal{F}\left(v_0\right) \ni \mu_0 \le -\kappa(t) + F = \kappa(\xi) + F,
		\end{align*}
		as $\kappa(t) \le \frac{2}{h}$.
		\item Passing through a cube\newline
		We can again translate (in space and time) and rotate (in space) everything back in $Q \coloneqq [-\frac{h}{2}, \frac{h}{2}]\times [0, h]$ and assume that $$\Gamma(0) \cap Q = \left\{ y = \sqrt{ \left(\tfrac{h}{2}\right)^2 - x^2} \right\}.$$
		Define in $Q$ the evolution
		\[
		\Gamma(t) \coloneqq \left\{ y = \sqrt{ \left(\tfrac{h}{2}\right)^2 - x^2} + v_0 t  \right\} \cup \{ |x| = \frac{h}{2}, y \le v_0 t \}
		\]
		for $t \in (0, \frac{h}{v_0})$ and it holds $\sup\mathcal{F}\left(v_n(\xi)\right) = \sup\mathcal{F}\left(v_0\right)\ni \mu_0 \le \kappa(\xi) + F$.
		\item Rotating through a cube\newline
		This is actually a combination of (2) followed by (1). First, we use the same construction as in (2) but cap the interface at the upper boundary of the cube, i.~e., we replace the evolution by
		\[
		\Gamma(t) \coloneqq \left\{ y = \min\left\{\sqrt{ \left(\tfrac{h}{2}\right)^2 - x^2} + v_0 t, h\right\}  \right\} \cup \{ |x| = \frac{h}{2}, y \le v_0 t \}
		\]
		if $t \le \frac{h}{v_0}$. This is still a subsolution and at $t= \frac{h}{v_0}$ the interface is a square and we can apply (1) on the flat face where the next cube in the path is to reinitialise the flow.
	\end{enumerate}
	Using this procedure, we can construct a solution of positive velocity, however to obtain a solution that leaves compact sets, we have to assure that each step in the construction of the solution only requires a finite amount of time. This is true, if $\frac{h}{v_0} < \infty$, i.~e., $v_0 > 0$ which is true by assumption.
\end{proof}

\begin{theorem}[Upper Bound]\label{thm:upper_bound}
	There is a constant $c(r_0) > 0$ such for each
	\[
	F >  \sup\mathcal{F}(0) + \frac{4}{(-\frac{1}{5}\log(p_c))^{\frac{1}{2}}} \sqrt{\dens}
	\]
	there exists a ballistic subsolution $\Gamma(t)$ with $\Gamma(0) = \{y=0\}$.
\end{theorem}
\begin{proof}
	This follows by applying lemma \ref{lem:local_propagating_subsolution} to each path from lemma \ref{lem:paths_with_density}.
\end{proof}

\begin{theorem}[Level-set subsolution]\label{thm:levelset_sub}
	Assume that $\theta > 0$ is small enough, then under the assumptions of theorem \ref{thm:upper_bound} there exists almost surely a ballistic level-set subsolution to equation \eqref{eq:main_levelset}.
\end{theorem}

\begin{proof}
	As in the proof of theorem \ref{thm:levelset_super} we can use the same construction to obtain a level-set supersolution. We have to choose $\theta$ so small that in a $\theta$ neighborhood of the solution curve from theorem \ref{thm:upper_bound} there are no obstacles.
\end{proof}

\section*{Acknowledgement}
LC acknowledges support from the Fonds National de la Recherche, Luxembourg (AFR Grant 13502370). PWD gratefully acknowledges partial support from the Deutsche Forschungsgemeinschaft (Grant No.~Do 1412/4-1 within SPP 2265). MO gratefully acknowledges support from the Deutsche Forschungsgemeinschaft (DFG, German Research Foundation) {\sl via} project 211504053 - SFB 1060 and project 390685813 -  GZ 2047/1 - HCM.

\section*{Declaration of Interest}
Declarations of interest: none.

\bibliographystyle{abbrv}
\bibliography{bib}

\end{document}